\documentclass{amsart}
\usepackage{graphicx}
\usepackage{amssymb, amsmath}
\usepackage{tikz}
\usepackage{cleveref}
\usepackage{bm}
\usepackage{xfrac}

\newtheorem{theorem}{Theorem}[section]
\newtheorem{corollary}[theorem]{Corollary}
\newtheorem{lemma}[theorem]{Lemma}
\newtheorem{proposition}[theorem]{Proposition}

\theoremstyle{definition}
\newtheorem{definition}[theorem]{Definition}

\newtheorem{example}[theorem]{Example}

\theoremstyle{remark}
\newtheorem{remark}[theorem]{Remark}

\numberwithin{equation}{section}

\newcommand{\aplim}{\operatorname*{ap-lim}}
\newcommand{\bound}{N}
\newcommand{\clos}{\operatorname{clos}}
\newcommand{\distanceM}{\operatorname{\mathcal{D}}}
\newcommand{\Haus}{\mathrm{Haus}}
\newcommand{\GH}{\mathrm{GH}}

\newcommand{\Hom}{\operatorname{Hom}}
\renewcommand{\Im}{\operatorname{Im}}
\newcommand{\Ker}{\operatorname{Ker}}
\newcommand{\Mag}{\mathrm{Mag}}

\newcommand{\FMetmag}{\operatorname{FMet}^\circ}
\newcommand{\sumofentry}{\operatorname{\mathbf{sum}}}
\newcommand{\adj}{\operatorname{adj}}

\newcommand{\diag}{\operatorname{diag}}
\newcommand{\dis}{\operatorname{dis}}
\newcommand{\rank}{\operatorname{rank}}

\newcommand{\Sym}{\operatorname{Sym}}
\newcommand{\trans}{\mathsf{T}}

\newcommand{\vol}{\operatorname{vol}}
\newcommand{\zetaM}{\operatorname{\mathcal{Z}}}

\newcommand{\Z}{\mathbb{Z}}

\newcommand{\R}{\mathbb{R}}

\title[Approximate continuity of magnitude and weighting]{Approximate Gromov--Hausdorff continuity of magnitude and weighting}

\author{Masahiko Yoshinaga}
\address{Department of Mathematics, The University of Osaka, Japan}
\email{yoshinaga@math.sci.osaka-u.ac.jp}

\begin{document}

\begin{abstract}
Magnitude gives an effective size of a finite metric space and is now used in several data-analysis settings. For such applications, it is natural to ask how magnitude behaves under Gromov-Hausdorff perturbations, including collisions of points. 
However, magnitude is nowhere continuous
on finite metric spaces with the Gromov--Hausdorff topology. We show that this
failure is nongeneric in a precise measure-theoretic sense. 
After fixing the number of points in each collapsing cluster, magnitude is approximately continuous. 

More strongly, when clusters of points collapse to the points of a
limit space, the total weighting of each cluster approximately converges to
the weighting of the corresponding limit point. The main tool is a general
result on approximate limits of inverses near singular matrices.
\end{abstract}

\maketitle

\tableofcontents


\section{Introduction}

\subsection{Effective number of points and magnitude}
\label{subsec:eff}
The invariant now called the magnitude of a finite metric space goes back to
Solow and Polasky's work on biological diversity \cite{sp-bio}. Their aim was
to measure an \emph{effective number of species} from pairwise dissimilarities:
species that are very similar should not count as fully distinct. The same
matrix formula appears in Leinster's definition of magnitude, which arose from
the Euler characteristic of enriched categories \cite{lei-mag}. Magnitude has since found connections with geometry and analysis \cite{lm-mag,mec-pos}, and magnitude homology was later introduced as its categorification \cite{hw-cat,ls-mh}. 

Let $X=\{x_1,\ldots,x_n\}$ be a finite metric space and put
$d_{ij}=d(x_i,x_j)$. Given a decreasing function
$f:[0,\infty)\to[0,1]$ with $f(0)=1$ and $f(a)\to0$ as $a\to\infty$, one can
form the \emph{similarity matrix} 
\[
 Z_X^{(f)}=(f(d_{ij}))_{i,j=1}^n.
\]
When this matrix is invertible, following 
the viewpoint of Solow and Polasky, we regard 
the sum of the entries of 
$(Z_X^{(f)})^{-1}$ as an \emph{effective number of points (``species'' in \cite{sp-bio})}. 
Write $\sumofentry(A)$ for the sum of all entries of a matrix $A$. Two basic
requirements are
\begin{align*}
\text{(a)}\quad & d_{ij}\gg0\quad(i\ne j)
&&\Longrightarrow& \sumofentry((Z_X^{(f)})^{-1})&\approx n,\\
\text{(b)}\quad & d_{ij}\approx0\quad\text{for all }i,j
&&\Longrightarrow& \sumofentry((Z_X^{(f)})^{-1})&\approx1.
\end{align*}

Condition (a) says that well-separated points should count as distinct, whereas condition (b) concerns what happens when several points collapse to a single point.

A related principle appears in the diversity theory of Leinster--Cobbold \cite{lc-meas}. Their diversity measures depend on both a similarity matrix and a prescribed abundance distribution. 
They are unchanged when identical species are merged, with their abundances added. Moreover, for a fixed number of species, their diversity measures depend continuously on the similarity data, so merging nearly identical species changes the diversity only slightly.

For magnitude, the behavior under such collisions is more subtle. If the number of points is fixed and the zeta matrix remains nonsingular, magnitude varies continuously with the distances. At a collision, however, several points may merge into one, and the corresponding limiting zeta matrix becomes singular. Thus the usual continuity argument for a fixed number of points does not apply. This leads naturally to the question of continuity with respect to the Gromov--Hausdorff topology, which allows such collisions.

From now on we use the standard choice $f(a)=e^{-a}$ and write
\[
 Z_X=(e^{-d_{ij}})_{i,j=1}^n
\]
for the \emph{zeta matrix}. If $Z_X$ is invertible, the \emph{weighting} and
\emph{magnitude} of $X$ are
\[
 \bm w_X=Z_X^{-1}J_{n,1},\qquad
 \Mag(X)=J_{1,n}Z_X^{-1}J_{n,1}=\sum_{i=1}^n w_X(x_i),
\]
where $J_{m,n}$ denotes the $m\times n$ all-ones matrix. For $t>0$, let
$tX=(X,td)$.

\subsection{Gromov--Hausdorff continuity and applications}

This question also fits into the broader study of quantitative relations between the Gromov-Hausdorff distance and metric invariants; see, for example, M\'emoli \cite{mem-gh}, motivated in part by problems in shape matching. 
It is also relevant to recent applications of magnitude in data analysis. Magnitude weightings have been used for image analysis and edge detection \cite{adamer-image}, while magnitude has been used to study neural network representations \cite{andreeva-nn} and the diversity of latent representations \cite{limbeck-div}. In the latter work, stability under data perturbations is itself one of the desired properties. The Gromov--Hausdorff metric is particularly natural in this context because it compares finite metric spaces without fixing labels or cardinality.

The simplest example of such a collision is obtained by scaling a fixed finite metric space to a point.
\begin{definition}
A finite metric space $X$ has the \emph{one-point property} if
\[
 \lim_{t\to0}\Mag(tX)=1.
\]
\end{definition}
This property can fail. Willerton's six-point example has small-scale limit
$6/5$ \cite[Example~2.2.8]{lei-mag}. It was proved that the
one-point property holds on a dense open set of finite metric spaces, while
exceptional small-scale limits can realize every prescribed value at least
$1$ \cite{ry-one}; examples with limits below $1$ are also known \cite{kam}. More general
paths of finite metric spaces collapsing to a point can behave even worse,
even for five-point spaces \cite[Example 2.6]{kry-con}. 
Several recent works have further studied continuity and monotonicity properties of magnitude 
\cite{gt-mon, hiy, kl-tra, kl-con, rw-mic}. 

Let
\[
 \FMetmag=\{X: X\text{ is a finite metric space and }\det Z_X\ne0\}.
\]
For a fixed cardinality, magnitude is continuous in the entries of the
distance matrix. The problem begins when Gromov--Hausdorff convergence allows
points to collide and the cardinality to drop. In fact, magnitude is nowhere
continuous on $\FMetmag$ in the Gromov--Hausdorff topology
\cite[Theorem~2.5]{kry-con}. 
On the other hand, magnitude does converge along generic linear paths approaching a fixed finite metric space \cite[Theorem 4.1]{kry-con}. 
This suggests that the discontinuity may be small in a measure-theoretic
sense, even though ordinary continuity fails everywhere.

We make this idea precise using approximate continuity. There is no single
Lebesgue measure on all finite metric spaces, since the space of ordered
$N$-point distance matrices has dimension $N(N-1)/2$. We therefore work
locally with fixed cardinality and fixed collision multiplicities. Each such
stratum is finite-dimensional and carries the usual Lebesgue density. Our
main theorem says that magnitude is approximately continuous on every such
stratum. We prove more: if clusters of points collapse to the points of a
limit space, then the total weighting of each cluster approximately converges
to the weighting of the corresponding limit point
(Theorem~\ref{thm:MMsplim}). 
Thus the exceptional approaches responsible for the failure of ordinary Gromov-Hausdorff continuity have density zero 
in every fixed local stratum.

The proof comes from linear algebra. When an invertible matrix $G$ approaches
a singular matrix $F$, the entries of $G^{-1}$ usually diverge. Nevertheless,
some linear combinations of these entries have canonical approximate limits.
Theorem~\ref{thm:linalg} gives a general criterion. The singular matrices
created by collisions of points satisfy this criterion, and this is why sums
of weights over collapsing clusters are more stable than the individual
weights.

\subsection{A motivating matrix calculation}

Condition~(b) suggests the following matrix statement. If an invertible
matrix $A$ is close to the all-ones matrix $J_{n,n}$, then one expects
\[
 \sumofentry(A^{-1})\approx1.
\]
The ordinary limit is false, even for $n=2$. For
\[
A_t=
\begin{pmatrix}
1+t+kt^2&1+2t\\
1+2t&1+3t
\end{pmatrix},
\]
one has $A_t\to J_{2,2}$ and
\[
 \sumofentry(A_t^{-1})\longrightarrow\frac{k}{k-1}
\]
as $t\to0$, whenever $k\ne1$. Thus different exceptional paths can give
different limits. The correct statement is obtained by ignoring a set of
approaches of density zero. In Corollary~\ref{cor:1pt} we prove
\[
 \aplim_{A\to J_{n,n}}\sumofentry(A^{-1})=1.
\]
Theorem~\ref{thm:linalg} proves a more general statement for inverses near
singular matrices.

For metric spaces, Gromov--Hausdorff convergence with bounded cardinality can
be described by distance matrices near a fixed limit space
(Lemma~\ref{lem:distmat}). Applying the matrix result to zeta matrices gives
approximate continuity of magnitude. The stronger weighting theorem says
that if a cluster collapses to a point $x_i$, then the sum of the weights in
that cluster approximately converges to $w_X(x_i)$. This is the multi-cluster
form of condition~(b).

The next example shows this behavior in a family where the limits can be
computed directly.

\begin{example}
\label{ex:limit}
Let $m, n\in\Z_{>0}$ and 
$\ell, a, b\in\R_{>0}$ such that $2\ell\geq\max\{a, b\}$. 
Define a metric on 
$Y_{a, b, \ell}=\{x_1, \dots, x_m, y_1, \dots, y_n\}$ 
by 
\begin{equation}
\begin{split}
d(x_i, x_j)&=a, (1\leq i, j\leq m, i\neq j)\\
d(y_i, y_j)&=b, (1\leq i, j\leq n, i\neq j)\\
d(x_i, y_j)&=\ell, (1\leq i\leq m, 1\leq j\leq n). 
\end{split}
\end{equation}
(See Figure \ref{fig:YtoX}.) 
\begin{figure}[htbp]
\centering
\begin{tikzpicture}

\filldraw[draw=black, fill=black] (0,0) circle [radius=0.06]; 
\filldraw[draw=black, fill=black] (0,0.4) node[above] {$x_j$} circle [radius=0.06]; 
\filldraw[draw=black, fill=black] (-0.35,0.2) node[left] {$x_i$} circle [radius=0.06]; 

\draw[thick, <->] (-0.4, -0.4) -- node[below] {$a$} ++(0.4, 0);

\filldraw[draw=black, fill=black] (8,0) circle [radius=0.06]; 
\filldraw[draw=black, fill=black] (8.6,0) circle [radius=0.06]; 
\filldraw[draw=black, fill=black] (8.6,0.6) node[above] {$y_j$} circle [radius=0.06]; 
\filldraw[draw=black, fill=black] (8,0.6) node[above] {$y_i$} circle [radius=0.06]; 

\draw[thick, <->] (8, -0.4) -- node[below] {$b$} ++(0.6, 0);

\draw[thick, <->] (0.4, -0.4) -- node[below] {$\ell$} ++(7.2, 0);

\end{tikzpicture}
\caption{$Y_{a, b, \ell}$}
\label{fig:YtoX}
\end{figure}
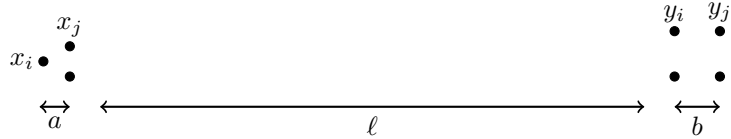
Then weightings are 
\[
\begin{split}
w(x_i)&=\frac{1+(n-1)e^{-b}-ne^{-\ell}}{(1+(m-1)e^{-a})(1+(n-1)e^{-b})-mne^{-2\ell}},\\
w(y_i)&=\frac{1+(m-1)e^{-a}-me^{-\ell}}{(1+(m-1)e^{-a})(1+(n-1)e^{-b})-mne^{-2\ell}}.
\end{split}
\]
Then we have the following limit as $a, b\to 0$
\[
\begin{split}
\lim_{a, b\to 0}w(x_i)&=
\frac{1}{m(1+e^{-\ell})},\\
\lim_{a, b\to 0}w(y_j)&=
\frac{1}{n(1+e^{-\ell})}.
\end{split}
\]
Therefore the total weight of each collapsing cluster converges to 
\[
\lim_{a, b\to 0}m\cdot w(x_i)=
\lim_{a, b\to 0}n\cdot w(y_j)=
\frac{1}{1+e^{-\ell}}, 
\]
which is the magnitude weighting of the two point space with 
distance $\ell$. 
Thus, although the individual weights depend on the number of points in each cluster, their sums have exactly the expected limit.
\end{example}

\section{Density and approximate continuity}
\label{sec:aplim}

\subsection{Density}

Let $A\subset\R^n$ be a Lebesgue measurable set and $\bm{a}\in\clos(A)$. 
\begin{definition}
\label{def:density}
If the limit 
\begin{equation}
\label{eq:density}
D(A, \bm{a}):=
\lim_{r\to 0}
\frac{\vol (A\cap B_r(\bm{a}))}{\vol (B_r(\bm{a}))}
\end{equation}
exists, it is called the \emph{density} of $A$ at $\bm{a}$ 
(where $B_r(\bm{a})=\{\bm{x}\in\R^n : |\bm{x}-\bm{a}|<r\}$ and 
$\vol(-)$ is the Lebesgue measure). 
\end{definition}

\begin{example}
Let 
$A=\{y>x^2\}$, $B=\{0<y<x^2\}$ and $C^\pm=\{\pm y>0\}$ 
as in Figure \ref{fig:density}. 
$A\cup C_-$ has density $1$ at $O$ in $\R^2$. 
$B$ has density $0$ at $O$ in $\R^2$. 
$A$ has density $1/2$ at $O$ in $\R^2$, but has density $1$ at $O$ in $C^+$. 
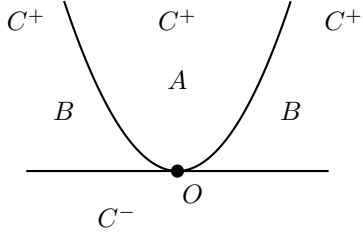
\begin{figure}[htbp]
\centering
\begin{tikzpicture}

\draw[thick] (-2,0) -- (2,0);

\draw[thick, domain=-1.5:1.5, samples=100, smooth] 
    plot (\x, {\x*\x});

\draw (0.2,-0.3) node {$O$}; 

\draw (0,1.2) node {$A$}; 
\draw (1.5,0.8) node {$B$}; 
\draw (-1.5,0.8) node {$B$}; 
\draw (-0.8,-0.6) node {$C^-$}; 
\draw (0,2) node {$C^+$}; 
\draw (-2,2) node {$C^+$}; 
\draw (2.2,2) node {$C^+$}; 

\filldraw[draw=black, fill=black] (0,0) circle [radius=0.08]; 

\end{tikzpicture}
\caption{Subsets of $\R^2$ and densities.}
\label{fig:density}
\end{figure}
\end{example}

\subsection{Approximate limit}

Let $U\subset\R^n$ be an open subset. 
Let us recall the following weaker 
notion of continuity \cite{eg-meas}.

\begin{definition}[Approximate limit; cf. \cite{eg-meas}]
Let $f:U\to\R$ be measurable and let $\bm a\in U$.  We write
\[
 \aplim_{\bm x\to\bm a}f(\bm x)=b
\]
if, for every $\varepsilon>0$, the measurable set
\[
 \{\bm x\in U:|f(\bm x)-b|<\varepsilon\}
\]
has density $1$ at $\bm a$. The function $f$ is \emph{approximately
continuous} at $\bm a$ if
\[
 \aplim_{\bm x\to\bm a}f(\bm x)=f(\bm a).
\]
\end{definition}

\begin{example}
We have 
\[
\aplim_{(x, y)\to(0,0)}\frac{x^2}{y}=0.
\]
It is a special case of Proposition \ref{prop:ratio} (1). 
\end{example}

\subsection{Two elementary invariance properties}

We record two facts that will be used when passing to the cone of distance matrices and when changing from distance to zeta coordinates.

\begin{proposition}[Restriction to a positive-density subset]
\label{prop:restrict}
Let $U\subset\R^n$ be open, let $f$ be a measurable map from almost all points in $U$ to a metric space $M$, and suppose that $\aplim_{x\to a}f(x)=m$. If $A\subset U$ is measurable and $D(A,a)>0$, then the restriction of $f$ to $A$ has relative approximate limit $m$ at $a$.
\end{proposition}
\begin{proof}
For $\varepsilon>0$, put $U_\varepsilon=\{x:d_M(f(x),m)<\varepsilon\}$. Since $D(U_\varepsilon,a)=1$,
\[
 \frac{\vol((A\setminus U_\varepsilon)\cap B_r(a))}{\vol(A\cap B_r(a))}
 \le
 \frac{\vol(U_\varepsilon^c\cap B_r(a))/\vol(B_r(a))}
 {\vol(A\cap B_r(a))/\vol(B_r(a))}\longrightarrow0,
\]
because the denominator tends to $D(A,a)>0$ as $r\to 0$. Hence $U_\varepsilon\cap A$ has relative density one in $A$ at $a$.
\end{proof}

\begin{proposition}[Invariance under a $C^1$ change of coordinates]
\label{prop:diffeodensity}
Let $U,V\subset\R^n$ be open and let $\Phi:U\to V$ be a $C^1$-diffeomorphism. For every measurable $A\subset U$ and $a\in U$,
\[
 D(A,a)=0 \Longleftrightarrow D(\Phi(A),\Phi(a))=0,
 \qquad
 D(A,a)=1 \Longleftrightarrow D(\Phi(A),\Phi(a))=1.
\]
Consequently approximate limits are invariant under such a change of coordinates.
\end{proposition}
\begin{proof}
Since $d\Phi(a)$ is invertible and $d\Phi$ is continuous, there are constants $0<c<C<\infty$ such that, for all sufficiently small $r$,
\[
 B_{cr}(\Phi(a))\subset \Phi(B_r(a))\subset B_{Cr}(\Phi(a)).
\]
After shrinking the neighborhood if necessary, there are also $0<m<M<\infty$ with
\[
 m\le |\det d\Phi(x)|\le M.
\]
The change-of-variables formula and the ball inclusions show that a density-zero set at $a$ is carried to a density-zero set at $\Phi(a)$. Applying the same argument to $\Phi^{-1}$ gives the converse. Applying the density-zero statement to complements gives the assertion for density one. The final assertion follows by applying this to the sets on which a function is $\varepsilon$-close to its proposed approximate limit.
\end{proof}

\subsection{Approximate continuity of rational functions}
Let $F,G\in\R[x_1,\ldots,x_n]$ be nonzero polynomials, and write
\[
 F=f_p+f_{p+1}+\cdots,\qquad G=g_q+g_{q+1}+\cdots,
\]
where $f_i$ and $g_j$ are homogeneous and $f_p,g_q$ are nonzero. We call
$f_p$ and $g_q$ the lowest-degree homogeneous parts. Even when the ordinary
limit of $F/G$ at the origin does not exist, its approximate limit is often
determined by these lowest-degree parts.
\begin{proposition}
\label{prop:ratio}
Let $F=f_p+f_{p+1}+\cdots$ and $G=g_q+g_{q+1}+\cdots$, where $f_p$ and $g_q$ are the nonzero lowest-degree homogeneous parts.
\begin{enumerate}
\item If $p>q\ge0$, then
\[
 \aplim_{\bm x\to0}\frac{F(\bm x)}{G(\bm x)}=0.
\]
\item If $p=q$ and $f_p=g_p\ne0$, then
\[
 \aplim_{\bm x\to0}\frac{F(\bm x)}{G(\bm x)}=1.
\]
\item Let $0\in L\subset\R^n$ be a linear subspace. The same conclusions hold for the restrictions to $L$, provided that the restriction to $L$ of the lowest-degree part of the denominator is not identically zero. 
\end{enumerate}
\end{proposition}
\begin{proof}
We prove (1). Write $\bm x=r\theta$, with $r>0$ and $\theta\in S^{n-1}$. Since the zero set of the nonzero polynomial $g_q|_{S^{n-1}}$ has $(n-1)$-dimensional measure zero, for every $\delta>0$ there is $\eta>0$ such that $E=\{\theta\in S^{n-1} : |g_q(\theta)|\geq \eta\}$ satisfies 
\[
 \frac{\vol_{n-1}(E)}{\vol_{n-1}(S^{n-1})}>1-\delta. 
\]
For $\theta\in S^{n-1}$, we have 
\[
 F(r\theta)=r^p(f_p(\theta)+O(r)),\qquad
 G(r\theta)=r^q(g_q(\theta)+O(r)).
\]
Since $|g_q(\theta)|$ is bounded below by $\eta$ on $E$, 
for sufficiently small $r$, we have 
\[
 \left|\frac{F(r\theta)}{G(r\theta)}\right|\le C r^{p-q}
\]
for $\theta\in E$. 
Thus for every $\varepsilon>0$ the set on which $|F/G|<\varepsilon$ occupies at least a proportion $1-\delta$ of every sufficiently small ball about $0$. Since $\delta>0$ is arbitrary, the required density is one.

For (2), write
\[
 \frac{F}{G}-1=\frac{F-G}{G}.
\]
The lowest nonzero homogeneous part of $F-G$, if $F\ne G$, has degree strictly larger than $p=q$, so (1) applies. If $F=G$, the assertion is immediate. Part (3) is the same argument on the Euclidean space $L$; the hypothesis ensures that the restricted denominator has the asserted lowest degree.
\end{proof}

\section{Approximate limits of inverses near singular matrices}
\label{sec:linear}

We first isolate the linear-algebraic mechanism behind the results of this paper.
Let $V$ and $W$ be real vector spaces of the same finite dimension and let
$f\colon V\to W$ be linear.  We identify $\Hom(V,W)$ with a Euclidean space
after choosing bases; the notion of a set of density one, and hence approximate
limits, is independent of this choice.

\begin{theorem}\label{thm:linalg}
Let $f\colon V\to W$ be linear, let $\ell\in V^*$ satisfy
$\ell|_{\Ker f}=0$, and let $b\in\Im f$.  Choose $a\in V$ with $f(a)=b$.
Then, on the open set of invertible maps $g\colon V\to W$,
\[
  \aplim_{g\to f}\ell(g^{-1}b)=\ell(a).
\]
The right-hand side is independent of the choice of $a$.
\end{theorem}

\begin{proof}
Choose decompositions
\[
 V=V'\oplus\Ker f,\qquad W=\Im f\oplus W'
\]
and bases adapted to them.  Put $p=\rank f$ and $q=\dim\Ker f$.  In these
bases
\[
 F=\begin{pmatrix}C&0\\0&0\end{pmatrix},\qquad
 G=\begin{pmatrix}C+R&S\\T&U\end{pmatrix},
\]
where $C\in M_p(\R)$ is invertible and the entries of $R,S,T,U$ are local
coordinates centred at $F$.  Since $\ell$ annihilates $\Ker f$ and $b\in\Im f$,
it is enough by linearity to treat
$\ell=e_i^{\trans}$ and $b=Ce_j$, with $1\le i,j\le p$, 
where $e_i$ is the $i$th standard basis. 

Write
\[
 e_i^{\trans}G^{-1}Ce_j=\frac{P_{ij}(R,S,T,U)}{Q(R,S,T,U)},
 \qquad Q=\det G.
\]
The lowest-degree homogeneous part of $Q$ has degree $q$ and equals
\[
 Q_q=\det(C)\det(U).
\]
For the numerator, the degree-$q$ homogeneous part is
\[
 (P_{ij})_q
 =e_i^{\trans}\adj(C)Ce_j\,\det(U)
 =\delta_{ij}\det(C)\det(U).
\]
Consequently, if $i=j$, the numerator and denominator have the same nonzero
lowest-degree homogeneous part, so Proposition~\ref{prop:ratio}(2) gives
approximate limit $1$.  If $i\ne j$, the degree-$q$ part of the numerator
vanishes; hence the numerator has order strictly larger than $q$, and
Proposition~\ref{prop:ratio}(1) gives approximate limit $0$.  Thus
\[
 \aplim_{G\to F}e_i^{\trans}G^{-1}Ce_j=\delta_{ij},
\]
which proves the theorem.
\end{proof}

For applications to metric spaces we need a relative version.  Let $E$ be a
Euclidean space, $x_0\in E$, and let $L\subset E$ be an affine subspace through
$x_0$.  Approximate limits on $L$ are taken with respect to Lebesgue measure on
$L$.

\begin{proposition}[Relative form]\label{prop:linalg-relative}
In the notation and coordinates of the proof of Theorem~\ref{thm:linalg}, let
$L$ be an affine linear subspace of $\Hom(V,W)$ containing $f$.  Put $L_0:=L-f$, the linear subspace of directions parallel to $L$. If the
degree-$q$ homogeneous part $Q_q$ of $G\mapsto\det G$ at $f$ restricts to a
nonzero polynomial on $L_0$, then
\[
 \aplim_{\substack{g\to f\\g\in L}}\ell(g^{-1}b)=\ell(a).
\]
\end{proposition}

\begin{proof}
Translate by $f$ and restrict the numerator and denominator in the preceding proof to $L_0=L-f$. By assumption the denominator still has lowest degree $q$.  The same comparison
of lowest-degree parts, followed by Proposition~\ref{prop:ratio}(3) in the
case $i=j$ and Proposition~\ref{prop:ratio}(1) in the case $i\ne j$, proves the
claim.
\end{proof}

\begin{corollary}\label{cor:1pt}
For the all-ones matrix $J_{n,n}$,
\[
 \aplim_{G\to J_{n,n}}\sumofentry(G^{-1})=1.
\]
\end{corollary}
\begin{proof}
The row vector $J_{1,n}$ annihilates $\Ker J_{n,n}$, and
$J_{n,1}=J_{n,n}e_1$.  Apply Theorem~\ref{thm:linalg} to
$\ell=J_{1,n}$ and $b=J_{n,1}$.
\end{proof}

\begin{definition}\label{def:multisetsver}
For $A=(a_{ij})\in M_n(\R)$ and
$\bm p=(p_1,\ldots,p_n)\in\Z_{>0}^n$, put $N=\sum_i p_i$ and define
\[
(A)_{\bm p}:=
\begin{pmatrix}
a_{11}J_{p_1,p_1}&\cdots&a_{1n}J_{p_1,p_n}\\
\vdots&\ddots&\vdots\\
a_{n1}J_{p_n,p_1}&\cdots&a_{nn}J_{p_n,p_n}
\end{pmatrix}.
\]
For $1\le i\le n$, let $r_i\in\R^{1\times N}$ be the indicator row vector of
the $i$th block.
\end{definition}

\begin{theorem}\label{thm:suminv}
Let $A\in M_n(\R)$ be invertible, let $\bm p\in\Z_{>0}^n$, and put
$J=(A)_{\bm p}$.  Write $w=A^{-1}J_{n,1}$.
\begin{enumerate}
\item For each $i$,
\[
 \aplim_{G\to J} r_iG^{-1}J_{N,1}=w_i.
\]
In particular,
\[
 \aplim_{G\to J}\sumofentry(G^{-1})=\sumofentry(A^{-1}).
\]
\item If $A$ is symmetric, the same conclusions hold when $G$ is restricted
to the affine space of symmetric matrices having the same diagonal as $J$.
\end{enumerate}
\end{theorem}

\begin{proof}
Choose $\widetilde w\in\R^N$ by placing $w_i$ in one chosen coordinate of the
$i$th block and zeros in all other coordinates.  Then
$J\widetilde w=J_{N,1}$.  Moreover, every $r_i$ annihilates $\Ker J$: indeed,
$J$ factors as $EAE^{\trans}$, where $E\colon\R^n\to\R^N$ repeats the $i$th
coordinate $p_i$ times and $E^{\trans}\colon\R^N\to\R^n$ takes block sums. 
Note that the $i$th row of $E^{\trans}$ is $r_i$. 
Since $A$ is invertible, $\Ker J=\Ker E^{\trans}$. 
Theorem~\ref{thm:linalg} therefore gives
\[
 \aplim_{G\to J}r_iG^{-1}J_{N,1}=r_i\widetilde w=w_i.
\]
Summing over $i$ proves the second assertion in (1).

For (2), let $\mathcal L$ be the affine space of symmetric matrices with the
same diagonal as $J$.  The rank of $J$ is $n$, so the lowest-degree part of
$\det(J+H)$ has degree $N-n$.  It remains, by
Proposition~\ref{prop:linalg-relative}, to check that this homogeneous
polynomial is not identically zero on the zero-diagonal symmetric tangent
space of $\mathcal L$.  Take $H$ block diagonal, with $i$th diagonal block
$t_i(J_{p_i,p_i}-I_{p_i})$ (and interpret the block as $0$ if $p_i=1$).
On the $(p_i-1)$-dimensional subspace of vectors in the $i$th block whose
coordinates sum to zero, $H$ acts as $-t_iI$.  On the complementary
$n$-dimensional subspace of block-constant vectors, $J+H$ tends to the
invertible map represented by $A\diag(p_1,\ldots,p_n)$ as $t_i\to0$ 
(with respect to the basis of 
block-indicator vectors). 
Consequently 
\[
 \det(J+H)=c\prod_{i=1}^n t_i^{p_i-1}+\text{terms of higher total degree}
\]
with $c=(-1)^{N-n}\det(A)\prod_{i=1}^np_i\ne0$.  
Thus the required restriction is nonzero, and the relative
form of the theorem applies to every $r_i$.  Summing again gives the magnitude
statement.
\end{proof}

\begin{remark}\label{rem:limwt}
The blockwise statement in Theorem~\ref{thm:suminv} is the key point for
weightings: although the individual coordinates of $G^{-1}J_{N,1}$ need not
converge when a block collapses, the sum of the coordinates in each collapsing
block has the canonical approximate limit $w_i$.
\end{remark}

\section{Convergence of metric spaces}

\label{sec:conti}

\subsection{Gromov-Hausdorff metric}
\label{subsec:ghm}

We recall the basic facts about the Gromov--Hausdorff metric \cite{bbi-cou,bh-non}. 
Let $(\mathcal{X}, d_{\mathcal{X}})$ be a metric space. 
Recall that the Hausdorff metric $d_\Haus$ between subsets 
$A$ and $B$ of $\mathcal{X}$ is 
\begin{equation}
d_\Haus(A, B)=\inf \{r>0\mid 
A\subset N_r(B) \mbox{ and } B\subset N_r(A)\}, 
\end{equation}
where 
$N_r(A)=\{\bm{x}\in\mathcal{X}\mid d_{\mathcal{X}}(x, A)<r\}$. 
In general, the Gromov-Hausdorff metric between two 
metric spaces $(X, d_X)$ and $(Y, d_Y)$ is 
\begin{equation}
d_{\GH}(X, Y)=
\inf_{X\stackrel{i}{\hookrightarrow}\mathcal{X}
\stackrel{j}{\hookleftarrow} Y}
d_\Haus(i(X), j(Y)), 
\end{equation}
where the infimum runs over all isometric embeddings $i:X\to\mathcal{X}$ and $j:Y\to\mathcal{X}$ into a common metric space $\mathcal{X}$. 
(We omit embeddings $i$ and $j$ in the sequel.) 

Let $p_1:X\times Y\to X$ (resp. $p_2:X\times Y\to Y$) 
be the first (resp. second) projection. 
A subset $\mathcal{R}\subset X\times Y$ is called a 
\emph{correspondence} if the restrictions 
$p_1|_{\mathcal{R}}:\mathcal{R}\to X$ and 
$p_2|_{\mathcal{R}}:\mathcal{R}\to Y$ are both surjective. 

The \emph{distortion} of $\mathcal{R}$ is defined by 
\begin{equation}
\dis(\mathcal{R})=\sup
\{|d_X(x, x')-d_Y(y, y')| : 
(x, y), (x', y')\in\mathcal{R}\}. 
\end{equation}
The distortion is related to the Gromov-Hausdorff metric 
by the following formula \cite[Theorem 7.3.25]{bbi-cou}
\begin{equation}
\label{eq:GHinfdis}
d_{\GH}(X, Y)=\frac{1}{2}\cdot\inf\{\dis(\mathcal{R})\mid 
\mathcal{R}\subset X\times Y\mbox{ is a correspondence}\}. 
\end{equation}
Let us denote by $D_X=(d_X(x, x'))_{x, x'\in X}$ the 
distance matrix of $X$. We will describe the Gromov-Hausdorff 
convergence of a sequence of finite metric spaces in terms 
of $D_X$. 

\subsection{Spaces of distance and zeta matrices}
\label{subsec:distzetasp}

We now describe Gromov--Hausdorff convergence to a fixed finite metric space $X$ in terms of distance matrices, under a bound on the cardinalities of the approximating spaces.


Recall that for an ordered metric space $X$, 
the \emph{distance matrix} $D_X$ is 
\begin{equation}
D_X=(d_X(x_i, x_j))_{i, j=1, \dots, n}. 
\end{equation}
The distance matrix depends on the ordering. Let $\rho:[n]\to[n]$ be a permutation of 
$[n]=\{1, \dots, n\}$. Then 
$X=\{x_{\rho(1)}, \dots, x_{\rho(n)}\}$ determines another 
distance matrix which we denote by 
\begin{equation}
D_{\rho(X)}=(d_X(x_{\rho(i)}, x_{\rho(j)}))_{i, j=1, \dots, n}. 
\end{equation}
The set of all distance matrices of $n$-point 
metric spaces is denoted by 
\begin{equation}
\distanceM_n=\{D_X\mid X=\{x_1, \dots, x_n\}\mbox{ is 
a metric space}\}. 
\end{equation}
More explicitly, $\distanceM_n$ has the following 
presentation. 
\begin{equation}
\distanceM_n=
\left\{(d_{ij})_{i, j=1, \dots, n}\left| \,
\begin{aligned}
&d_{ii}=0, d_{ij}=d_{ji}, (1\leq i, j\leq n)\\
&d_{ij}>0, (1\leq i, j\leq n, i\neq j)\\
&d_{ik}\leq d_{ij}+d_{jk}, (i\neq j\neq k\neq i)
\end{aligned}
\right.
\right\}. 
\end{equation}
Let $\Sym_n^0=\{(d_{ij})_{i, j=1, \dots, n}\mid 
d_{ij}=d_{ji}, d_{ii}=0\}$ be the set of symmetric 
matrices with zero diagonal entries. Then, clearly, 
$\distanceM_n\subset\Sym_n^0$, furthermore, 
$\distanceM_n$ is a full-dimensional cone in $\Sym_n^0$. 
The closure of $\distanceM_n$ in $\Sym_n^0$ is 
\begin{equation}
\overline{\distanceM_n}=
\left\{(d_{ij})_{i, j=1, \dots, n}\left| \,
\begin{aligned}
&d_{ii}=0, d_{ij}=d_{ji}, (1\leq i, j\leq n)\\
&d_{ij}\geq 0, (1\leq i, j\leq n, i\neq j)\\
&d_{ik}\leq d_{ij}+d_{jk}, (i\neq j\neq k\neq i)
\end{aligned}
\right.
\right\}. 
\end{equation}
We will consider ($\frac{n(n-1)}{2}$-dimensional) 
Lebesgue measure on $\Sym_n^0$ through the 
natural identification $\Sym_n^0\simeq\R^{n(n-1)/2}$. 
We will use the measure to discuss approximate 
continuity with respect to the limit $Y\to X$. 
Since the defining inequalities of $\distanceM_n$ 
are linear, we have the following. 
\begin{proposition}
\label{prop:symmat}
$\distanceM_n$ has a positive density at every point 
$D\in\overline{\distanceM_n}$ in $\Sym_n^0$.
\end{proposition}
\begin{proof}
The closure $\overline{\distanceM_n}$ is a full-dimensional convex polyhedral cone in $\Sym_n^0$. Fix, for example, an equilateral distance matrix $D_0$, which is an interior point of $\distanceM_n$. For every $D\in\overline{\distanceM_n}$, convexity implies that the segment $(D,D_0]$ lies in the interior. 
Since the cone is defined by finitely many linear inequalities, the ratio of the volume of this cone to that 
of the ball centered at $D$ with sufficiently small radius is constant.  
Hence the density exists and equals the positive solid-angle ratio of this tangent cone.
\end{proof}
Similarly, we denote the set of zeta matrices by 
\begin{equation}
\begin{split}
\zetaM_n&=\{Z_X\mid X=\{x_1, \dots, x_n\}\mbox{ is 
a metric space}\}\\
&=
\left\{(q_{ij})_{i, j=1, \dots, n}\left| \,
\begin{aligned}
&q_{ii}=1, q_{ij}=q_{ji}, (1\leq i, j\leq n)\\
&0<q_{ij}<1, (1\leq i, j\leq n, i\neq j)\\
&q_{ik}\geq q_{ij}\cdot q_{jk}, (i\neq j\neq k\neq i)
\end{aligned}
\right.
\right\}. 
\end{split}
\end{equation}
Let 
\begin{equation}
\Sym_n^1=\{(q_{ij})_{i, j=1, \dots, n}\mid 
q_{ij}=q_{ji}>0, q_{ii}=1\}
\end{equation}
be the set of symmetric matrices 
with positive entries such that all 
diagonal entries are $1$. 
The componentwise exponential 
$(d_{ij})_{ij}\longmapsto (q_{ij}=\exp(-d_{ij}))_{ij}$ 
gives a diffeomorphism 
\begin{equation}
\Sym_n^0\stackrel{\simeq}{\longrightarrow}\Sym_n^1, 
\end{equation}
which sends $\distanceM_n$ bijectively to $\zetaM_n$. 
Hence Proposition~\ref{prop:diffeodensity} shows that density-zero and density-one statements, and therefore approximate limits, are equivalent in the two coordinate systems.

\subsection{Distance matrices near a limit}

Fix $X=\{x_1,\ldots,x_n\}\in\FMetmag$ and an integer $\bound\ge n$, and put
\[
 \sigma=\frac14\min\{d_X(x,x'):x\ne x'\}.
\]
Let $Y\in\FMetmag$ satisfy $\#Y\le\bound$ and $d_{\GH}(X,Y)<\sigma$.
Choose isometric embeddings of $X$ and $Y$ into a common metric space such
that their Hausdorff distance is less than $\sigma$. For each $i$, let
\[
 Y_i=Y\cap N_\sigma(x_i).
\]
Then
\begin{equation}
\label{eq:Ydec}
 Y=Y_1\sqcup\cdots\sqcup Y_n.
\end{equation}
Moreover, two points $y,y'\in Y$ lie in the same $Y_i$ if and only if
$d_Y(y,y')<2\sigma$. Indeed, points in one cluster are less than $2\sigma$
apart, while points in different clusters are more than $2\sigma$ apart.
Thus the decomposition~\eqref{eq:Ydec} is independent of the chosen
embeddings.

Let $s_i=\#Y_i$. We call the vector 
$\bm{s}=(s_1, \dots, s_n)\in(\Z_{>0})^n$ 
the \emph{multiplicity vector}. 
For each $i=1, \dots, n$, fix an ordering of elements 
$Y_i=\{y^i_1, \dots, y^i_{s_i}\}$ and consider associated 
total ordering 
$Y=\{y^1_1, y^1_2, \dots, y^1_{s_1}, 
y^2_1, \dots, y^2_{s_2}, \dots, 
y^n_1, \dots, y^n_{s_n}\}$ on $Y$. 

The following lemma relates Gromov-Hausdorff distance to the distance-matrix coordinates introduced above. 


\begin{lemma}
\label{lem:distmat}
Let $X,Y,n,\bound,\sigma,$ and $\bm{s}$ be as above. 
Let $0<\varepsilon<\sigma$. 
Then $d_{\GH}(X,Y)<\varepsilon$ if and only if there is a permutation
$\rho$ of $[n]$ such that
\[
 \|D_Y-(D_{\rho(X)})_{\bm s}\|_\infty<2\varepsilon,
\]
where $\|(a_{ij})\|_\infty=\max_{i,j}|a_{ij}|$. 
\end{lemma}
\begin{proof}
Suppose $d_{\GH}(X, Y)<\varepsilon$. Then there exist 
isometric embeddings 
$X\hookrightarrow\mathcal{X}\hookleftarrow Y$ such that 
$Y\subset N_{\varepsilon}(X)$. The decomposition $Y=\bigsqcup_{i=1}^n (Y\cap N_{\varepsilon}(x_i))$ agrees with~\eqref{eq:Ydec} up to relabeling. Hence 
there exists a permutation $\rho$ of $[n]$ such that 
$Y_i=Y\cap N_{\varepsilon}(x_{\rho(i)})$. Then from 
$y^i_\alpha\in N_\varepsilon(x_{\rho(i)})$ and 
$y^j_\beta\in N_\varepsilon(x_{\rho(j)})$, we have 
\begin{equation}
\label{eq:diff2eps}
|d_Y(y^i_\alpha, y^j_\beta)-d_X(x_{\rho(i)}, x_{\rho(j)})|
<2\varepsilon, 
\end{equation}
for any 
$i, j=1, \dots, n, 1\leq\alpha\leq s_i, 1\leq\beta\leq s_j$. 
Thus we have $\|D_Y-(D_{\rho(X)})_{\bm s}\|_\infty<2\varepsilon$. 

Conversely, suppose 
$\|D_Y-(D_{\rho(X)})_{\bm s}\|_\infty<2\varepsilon$ 
for some permutation $\rho$. It is equivalent to 
(\ref{eq:diff2eps}). Then the correspondence 
\begin{equation}
\label{eq:corresp}
\mathcal{R}=
\bigsqcup_{i=1}^n\{x_{\rho(i)}\}\times Y_i\subset X\times Y. 
\end{equation}
satisfies $\dis(\mathcal{R})<2\varepsilon$. Then by 
(\ref{eq:GHinfdis}), $d_{\GH}(X, Y)<\varepsilon$. 
\end{proof}

\subsection{Stratified approximate continuity of magnitude and weighting}
\label{subsec:appconti}

There is no canonical Lebesgue measure on the union of all cardinality strata
of the Gromov--Hausdorff space of finite metric spaces: the stratum of
$N$-point spaces has dimension $N(N-1)/2$.  We therefore formulate approximate
continuity separately on each local multiplicity stratum.  This is also the
form naturally suggested by Lemma~\ref{lem:distmat}.

Fix $X=\{x_1,\ldots,x_n\}\in\FMetmag$ and
\[
 \sigma=\frac14\min_{i\ne j}d_X(x_i,x_j).
\]
Let $\bm s=(s_1,\ldots,s_n)\in\Z_{>0}^n$ and $N=\sum_i s_i$.  For an ordered
$N$-point metric space $Y$ sufficiently close to $X$ and having multiplicity
vector $\bm s$, order its points blockwise as in the preceding subsection.  In
these coordinates
\[
 D_Y\longrightarrow (D_X)_{\bm s}\quad\text{and}\quad
 Z_Y\longrightarrow (Z_X)_{\bm s}.
\]
We equip the set of such distance matrices with the restriction of Lebesgue
measure from $\Sym_N^0$.  By Proposition~\ref{prop:symmat}, this set has
positive density at $(D_X)_{\bm s}$.  Permuting points within blocks, or
permuting the blocks together with the labels of $X$, changes coordinates by a
linear isometry and therefore does not affect density-one statements.

\begin{definition}\label{def:strat-ap}
Let $F$ be a function defined almost everywhere on the above multiplicity stratum to a metric space $M$. 
We say that
$F$ has \emph{stratified approximate limit} $m\in M$ at $X$ along
$\bm s$ if, in the distance-matrix coordinates just described, for every
$\varepsilon>0$ the set
\[
 \{D_Y:d_M(F(Y),m)<\varepsilon\}
\]
has relative density one at $(D_X)_{\bm s}$ inside $\distanceM_N$.
\end{definition}

The componentwise map $D=(d_{ij})\mapsto (e^{-d_{ij}})$ is a smooth
diffeomorphism from $\Sym_N^0$ 
onto the set of symmetric matrices with positive entries and diagonal entries $1$. 
By Proposition~\ref{prop:diffeodensity}, this diffeomorphism preserves density-zero and density-one sets locally. Thus we may
compute the same approximate limits in zeta-matrix coordinates.

For $Y\in\FMetmag$, write $w_Y=Z_Y^{-1}J_{N,1}$.  If
$Y=Y_1\sqcup\cdots\sqcup Y_n$ is the cluster decomposition over $X$, define
the total weight of the $i$th cluster by 
\[
 m_i(Y):=\sum_{y\in Y_i}w_Y(y).
\]

\begin{theorem}[Approximate continuity of weighting]\label{thm:MMsplim}
Let $X=\{x_1,\ldots,x_n\}\in\FMetmag$ and let
$\bm s=(s_1,\ldots,s_n)\in\Z_{>0}^n$.  Then, for every $i=1,\ldots,n$,
\[
 \aplim_{\substack{Y\to X\\\text{along }\bm s}}m_i(Y)=w_X(x_i).
\]
Equivalently, for every $\varepsilon>0$, the set of $Y$ in the multiplicity
stratum $\bm s$ satisfying
\[
 \max_{1\le i\le n}\left|
 \sum_{y\in Y_i}w_Y(y)-w_X(x_i)\right|<\varepsilon
\]
has relative density one at $X$.
\end{theorem}

\begin{proof}
Put $A=Z_X$.  Since $X\in\FMetmag$, $A$ is invertible.  In blockwise
coordinates the zeta matrices satisfy
\[
 Z_Y\longrightarrow J:=(A)_{\bm s}.
\]
They lie in the affine space of symmetric matrices with diagonal $1$.
Theorem~\ref{thm:suminv}(2), applied to the indicator row vector $r_i$ of the
$i$th block, gives
\[
 \aplim_{Z_Y\to J}r_iZ_Y^{-1}J_{N,1}
 =(A^{-1}J_{n,1})_i=w_X(x_i).
\]
The left-hand side is precisely $m_i(Y)$.  The passage from the ambient affine
space to the metric cone is justified by Proposition~\ref{prop:restrict} and
Proposition~\ref{prop:symmat}. Finally,
the distance-to-zeta diffeomorphism preserves density-one sets.  This proves
the assertion.
\end{proof}

\begin{corollary}[Approximate continuity of magnitude]\label{thm:main}
Let $X\in\FMetmag$ and fix a multiplicity vector $\bm s$.  Then
\[
 \aplim_{\substack{Y\to X\\\text{along }\bm s}}\Mag(Y)=\Mag(X).
\]
Thus magnitude is approximately continuous at every finite metric space with
invertible zeta matrix, on every fixed local cardinality/multiplicity
stratum of the Gromov--Hausdorff space.
\end{corollary}

\begin{proof}
Since $\Mag(Y)=\sum_i m_i(Y)$ and
$\Mag(X)=\sum_iw_X(x_i)$, the result follows immediately from
Theorem~\ref{thm:MMsplim}.  Alternatively, it is the scalar statement in
Theorem~\ref{thm:suminv}(2).
\end{proof}

\begin{remark}
Corollary~\ref{thm:main} should not be confused with ordinary continuity.
Ordinary Gromov--Hausdorff continuity fails dramatically: magnitude is nowhere
continuous on the space of finite metric spaces \cite[Theorem~2.5]{kry-con}.
The result says instead that, after fixing the finite-dimensional stratum in
which the approximating spaces live, the directions responsible for the
failure of continuity form a set of asymptotic density zero.
\end{remark}

\begin{remark}
Our formulation of approximate continuity treats each fixed-cardinality stratum separately and does not give a uniform statement as the cardinality tends to infinity. In contrast, no cardinality bound is needed if one restricts to finite metric spaces admitting nonnegative weightings.
Recently, Hiyoshi \cite{hiy} proved that the logarithm of magnitude is Lipschitz continuous on this class: 
\[
|\log\Mag(X_1)-\log\Mag(X_2)|\leq 2d_{\GH}(X_1, X_2). 
\]
\end{remark}

\medskip

\emph{Acknowledgements.}
This work was partially supported by JSPS KAKENHI (Grant No. JP23H00081). 
ChatGPT (OpenAI) was used during the preparation of this article for mathematical discussion and language editing. The author is responsible for all assertions in this paper.

\end{document}